\documentclass[12pt,twoside,singlespace]{amsart}
\usepackage{array, paralist, enumerate, amsmath, amsfonts, amssymb, amscd, color, mathrsfs,comment}
\usepackage{amsthm} 

\usepackage{geometry}
\usepackage{framed}
\usepackage{hyperref}
\usepackage{graphicx}
\usepackage{epstopdf}
\usepackage[all,cmtip]{xy}
\usepackage{tikz}
\usepackage{tkz-graph}
\usetikzlibrary{arrows,%
                shapes,positioning}

\definecolor{DarkBlue}{rgb}{0,0,0.8} 
\definecolor{DarkGreen}{rgb}{0,0.5,0.0} 
\definecolor{DarkRed}{rgb}{0.9,0.0,0.0} 

\usepackage[T1]{fontenc}
\usepackage[latin1]{inputenc}

\newtheorem*{thm*}{Theorem}

\numberwithin{equation}{section}
\newtheorem{thm}[equation]{Theorem}
\newtheorem{lem}[equation]{Lemma}
\newtheorem{cor}[equation]{Corollary}

\theoremstyle{definition}

\newtheorem{ex}[equation]{Example}	
\newtheorem{remark}[equation]{Remark}

\newcommand{\ZZ}{\mathbf{Z}}

\title{A Unified Enumeration of $1$-dimension Garden Algebras and Valise Adinkras}
\author{Yan X Zhang}
\address{Yan X Zhang,
Dept. of Mathematics,
San Jose State University,
San Jose, CA 95192}
\email{yan.x.zhang@sjsu.edu}

\begin{document}

\pagestyle{plain}

\begin{abstract}
In the study of supersymmetry in one dimension, various works enumerate sets of generators of \emph{garden algebras} $GR(d,N)$ (and equivalently, \emph{valise Adinkras}) for special cases $N = d = 4$ and $N = d = 8$, using group-theoretic methods and computer computation. We complement this work by enumerating the objects for arbitrary $N$ and $d$ via a formula in a streamlined manner.
\end{abstract}

\maketitle

\section{Introduction}
\label{sec:introduction}
In a series of works by Gates et al. (\cite{gates:genomics}, \cite{bellucci2006journey}, \cite{chappell20134d}, etc.), the $GR(d,N)$ ``garden algebras'' and their visual representations ``valise Adinkras'' have been used to study supersymmetry in $1$ ($0$ spacial and $1$ time) dimension. Some of the recent research was on the enumeration and classification of these objects. Chapell, Gates, and H\"ubsch \cite{chappell2014adinkra} and Randall \cite{randall} used computers to enumerate these matrices for $N = d = 4$ and $N = d= 8$. More recently, Gates, H\"ubsch, Iga, and Mendez-Diez \cite{gates-counting} proved these results algebraically, without computer aid, by coset enumeration. The proofs used properties of the Klein-$4$ group and did not generalize directly to higher parameters. The main numerical results of \cite{gates-counting} can be seen in Figure~\ref{fig:previous}.

\begin{figure}
\begin{center}
\begin{tabular}{c|c|c}
& $N = d =4$ & $N = d =8$ \\
\hline
$\{|L_1|, \ldots, |L_N|\}$ & 6 & 151200 \\
$\{L_1, \ldots, L_N\}$ & 1536 & 79272345600 \\
\end{tabular}
\end{center}
\caption{Previous results in \cite{gates-counting}. The first row refers to sets of ``permutation generators,'' which are equivalence classes of sets of generators of garden algebras under ``forgetting'' of signs. The second row refers to sets of generators proper for the garden algebras. \label{fig:previous}}
\end{figure}

In this note, we give a unified enumeration of these objects for arbitrary $N$ and $d$ and recover the above results as special cases. The main idea is simply connecting some mathematical groundwork we laid in our previous work \cite{zhang:adinkras} and Gaborit's exploration of enumeration of codes \cite{gaborit1996mass}.  We assert that while our work gives a more general enumeration, it does not subsume these other works, which offer particular insights for the special cases that our work do not provide. Rather, we think of our work as a general complement to the existing work. We give some preliminaries in Section~\ref{sec:prelim} and our main proof in Section~\ref{sec:proof}.

\section{Preliminaries of Garden Algebras and Adinkras}
\label{sec:prelim}

We review the essentials of garden algebras and Adinkras. Adinkras are basically visual encodings of garden algebras, with some additional data. The interested reader should refer to \cite{d2l:first} for the introduction of Adinkras into the physics literature, or our \cite{zhang:adinkras} for a more mathematical treatment of Adinkras. In short, these objects encode a special class of representations of the supersymmetry algebra that have many nice properties. As this is a paper mainly about enumeration, we do not focus on the physics background.

A \emph{permutation} matrix is a square matrix with exactly one $1$ in each row and each column, with all other entries equalling $0$. A \emph{signed permutation matrix} is a square matrix with exactlye one nonzero element, which must be $1$ or $-1$, in each row and each column. We define the $GR(d,N)$ \emph{garden algebras} to be algebras generated by $N > 0$ $d \times d$ signed permutation matrices $\{L_1, \ldots, L_N\}$ which satisfy the relations:
\begin{align*}
  L_i L_j^T + L_j L_i^T & = 2 \delta_{ij} I \\
  L_i^T L_j + L_j^T L_i & = 2 \delta_{ij} I.
\end{align*}
Here, $I$ denotes the identity matrix and $T$ denotes transpose. Call such a $\{L_1, \ldots, L_N\}$ a \emph{set of generators} for $GR(d,N)$. We can similarly define a \emph{list of generators} $(L_1, \ldots, L_N)$ if we choose to remember the order of the elements, with the obvious $N!$ to $1$ correspondence between lists and sets. Given a signed permutation matrix $M$, let $|M|$ be the permutation matrix where every entry is $1$ if the corresponding entry of $M$ is $1$ or $-1$, and $0$ otherwise. We call the list $(|L_1|, \ldots, |L_N|)$ arising from a list of generators $(L_1, \ldots, L_N)$  an \emph{unsigned list of generators} of $M$. Unsigned lists of generators do not necessarily satisfy nice relations; they come up for the ease of classifying (signed) lists of generators.

\begin{ex}
\label{ex:matrices}
One possible list of generators for $GR(4,4)$ is $(L_1, L_2, L_3, L_4)$ respectively equalling
\[
\begin{bmatrix} 0 & 0 & 1 & 0 \\
0 & 0 & 0 & 1 \\
1 & 0 & 0 & 0 \\
0 & 1 & 0 & 0 \end{bmatrix}, \begin{bmatrix} 0 & -1 & 0 & 0 \\
1 & 0 & 0 & 0 \\
0 & 0 & 0 & -1 \\
0 & 0 & 1 & 0 \end{bmatrix},  \begin{bmatrix} 0 & 0 & 0 & 1 \\
0 & 0 & -1 & 0 \\
0 & -1 & 0 & 0 \\
1 & 0 & 0 & 0 \end{bmatrix}, \begin{bmatrix} 1 & 0 & 0 & 0 \\
0 & 1 & 0 & 0 \\
0 & 0 & -1 & 0 \\
0 & 0 & 0 & -1 \end{bmatrix}.
\]
We can check that the matrices obey the desired relations. The unsigned list of generators corresponding to this list is $(|L_1|, |L_2|, |L_3|, |L_4|)$ respectively equalling
\[
\begin{bmatrix} 0 & 0 & 1 & 0 \\
0 & 0 & 0 & 1 \\
1 & 0 & 0 & 0 \\
0 & 1 & 0 & 0 \end{bmatrix}, \begin{bmatrix} 0 & 1 & 0 & 0 \\
1 & 0 & 0 & 0 \\
0 & 0 & 0 & 1 \\
0 & 0 & 1 & 0 \end{bmatrix}, \begin{bmatrix} 0 & 0 & 0 & 1 \\
0 & 0 & 1 & 0 \\
0 & 1 & 0 & 0 \\
1 & 0 & 0 & 0 \end{bmatrix}, \begin{bmatrix} 1 & 0 & 0 & 0 \\
0 & 1 & 0 & 0 \\
0 & 0 & 1 & 0 \\
0 & 0 & 0 & 1 \end{bmatrix}.
\]
\end{ex}

An $N$-dimensional \emph{chromotopology} is a finite connected simple graph $G$ such that:

\begin{itemize}
\item $G$ is $N$-regular (every vertex has exactly $n$ incident edges) and bipartite;
\item The edges of $G$ are colored by $n$ colors such that every vertex is incident to exactly one edge of each color;
\item We assume the colors come with an ordering; that is, we can label the colors with the integers $1$ through $n$;
\item For any distinct colors $i$ and $j$, the edges in $G$ with colors $i$ and $j$ form a disjoint union of $4$-cycles. 
\end{itemize}

A canonical example of a chromotopology is an $n$-dimensional Hamming cube, where there are $2^n$ vertices labeled by the length-$n$ bitstrings, and the edges correspond to pairs of vertices with Hamming distance $1$. \emph{Adinkras}\footnote{All Adinkras in our paper refer to $1$-dimensional Adinkras working in $1$-d supersymmetry, so we suppress the dimensional adjective.} are defined to be chromotopologies with $2$ additional pieces of data: 
\begin{itemize}
\item A \emph{ranking} of a chromotopology $G$ is a map $h$ from the vertices of $G$ to $\ZZ$ that satisfies certain restraints. For our note, we limit ourselves to the \emph{valise ranking}, which simply means having $h(v) \in \{0,1\}$ for every vertex $v$ and having every edge $(x,y)$ in $G$ satisfy $\{h(x), h(y)\} = \{0,1\}$. In other words, a valise ranking is equivalent to a bipartition of $G$. We visualize this by putting the vertices into two rows, each row corresponding to one of the parts of the bipartition. Thus, a \emph{valise ranked chromotopology} (or \emph{valise Adinkra}) means a chromotopology (or Adinkra) with a valise ranking.
\item A \emph{dashing} of a chromotopology $G$ is a map $d$ from the edges of $G$ to $\ZZ_2$ such that the sum of $d(e)$ as $e$ runs over each $2$-colored $4$-cycle (that is, a $4$-cycle of edges that use a total of $2$ colors) is $1 \in \ZZ_2$; alternatively, every $2$-colored $4$-cycle contains an odd number of $1$'s. We typically draw a dash for each edge marked $1$ and a solid edge for each edge marked $0$.
\end{itemize}

\begin{remark}
When we restrict to valise rankings, our problem is fundamentally equivalent to some sort of classification of lists of generators of Clifford algebras/groups. Elements of $GR(d,N)$ ``basically'' satisfy the constraints of being generators for Clifford algebras. Generalizations of the relations to the more complicated supersymmetry algebra correspond to Adinkras with more complex ranking functions than the valise ranking.  
\end{remark}

\begin{figure}[htb]
\begin{center}

\begin{tabular}{c|c|c}
\begin{tikzpicture}[scale=0.10]
\SetVertexNoLabel
\SetUpEdge[labelstyle={draw}]
\Vertex[x=0,y=0]{111}
\Vertex[x=20,y=0]{011}
\Vertex[x=0,y=20]{101}
\Vertex[x=20,y=20]{001}
\Vertex[x=-10,y=10]{110}
\Vertex[x=10,y=10]{010}
\Vertex[x=-10,y=30]{100}
\Vertex[x=10,y=30]{000}
\Edge[color=red](100)(101)
\Edge[color=red](000)(001)
\Edge[color=red](010)(011)
\Edge[color=red](110)(111)
\Edge[color=green](000)(100)
\Edge[color=green](001)(101)
\Edge[color=green](010)(110)
\Edge[color=green](011)(111)
\Edge[color=blue](000)(010)
\Edge[color=blue](001)(011)
\Edge[color=blue](100)(110)
\Edge[color=blue](101)(111)
\end{tikzpicture}
&
\begin{tikzpicture}[scale=0.15]
\SetVertexNoLabel
\SetUpEdge[labelstyle={draw}]
\Vertex[x=0,y=20]{111}
\Vertex[x=0,y=10]{101}
\Vertex[x=20,y=20]{010}
\Vertex[x=20,y=10]{000}
\Vertex[x=-10,y=10]{110}
\Vertex[x=10,y=10]{011}
\Vertex[x=-10,y=20]{100}
\Vertex[x=10,y=20]{001}
\Edge[color=red](100)(101)
\Edge[color=red](000)(001)
\Edge[color=red](010)(011)
\Edge[color=red](110)(111)
\Edge[color=green](000)(100)
\Edge[color=green](001)(101)
\Edge[color=green](010)(110)
\Edge[color=green](011)(111)
\Edge[color=blue](000)(010)
\Edge[color=blue](001)(011)
\Edge[color=blue](100)(110)
\Edge[color=blue](101)(111)
\end{tikzpicture}
&
\begin{tikzpicture}[scale=0.15]
\SetVertexNoLabel
\SetUpEdge[labelstyle={draw}]
\Vertex[x=0,y=20]{111}
\Vertex[x=0,y=10]{101}
\Vertex[x=20,y=20]{010}
\Vertex[x=20,y=10]{000}
\Vertex[x=-10,y=10]{110}
\Vertex[x=10,y=10]{011}
\Vertex[x=-10,y=20]{100}
\Vertex[x=10,y=20]{001}
\Edge[color=red,style=dashed](100)(101)
\Edge[color=red](000)(001)
\Edge[color=red](010)(011)
\Edge[color=red,style=dashed](110)(111)
\Edge[color=green](000)(100)
\Edge[color=green](001)(101)
\Edge[color=green](010)(110)
\Edge[color=green](011)(111)
\Edge[color=blue, style=dashed](000)(010)
\Edge[color=blue](001)(011)
\Edge[color=blue, style=dashed](100)(110)
\Edge[color=blue](101)(111)
\end{tikzpicture}
\end{tabular}
\caption{Left to right: a chromotopology, a valise ranked chromotopology, an (valise) Adinkra, obtained in sequence by adding more and more structure. \label{fig:initial examples}}
\end{center}
\end{figure}
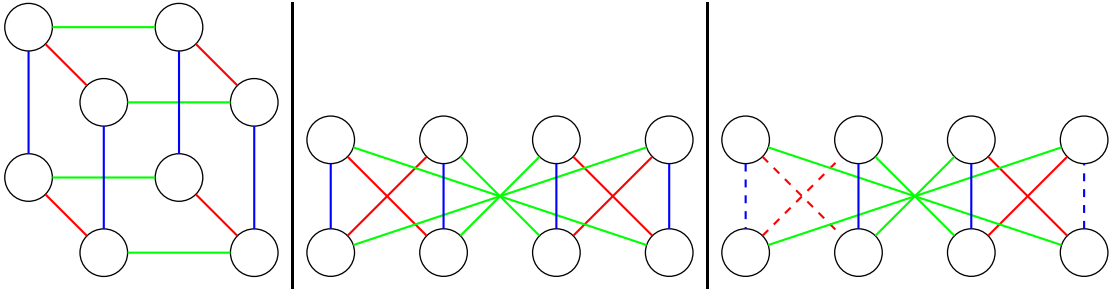

For an example of these objects, see Figure~\ref{fig:initial examples}. We call an Adinkra (or a ranked chromotopology) \emph{row-ordered} if it is equipped with an ordering of the vertices in each row of the ranking. Lists of generators of garden algebras and row-ordered Adinkras are related precisely by the following fact:
\begin{lem}
\label{lem:main equivalence}
There is a bijection between length-$N$ lists of generators $(L_1, \ldots, L_N)$ of $GR(d,N)$ and $N$-dimensional row-ordered valise Adinkras with $2d$ vertices.
\end{lem}
\begin{proof}[Proof Sketch]
Starting with a list of generators, we can obtain a row-ordered Adinkra via the following procedure: for each $\pm 1$ in row $j$ and column $k$ of matrix $L_i$, draw an edge of color $i$ (here, we use the orderings of the colors) from the $j$-th node in the top row to the $k$-th node in the bottom row. Dash the edge if the entry is $-1$ and draw a solid edge if the entry is $1$. We can now check that the relations of the $GR(d,N)$ algebra are in fact equivalent to the $2$-colored $4$-cycle condition of Adinkras. The converse is obvious from the construction. We omit the remaining technical details.
\end{proof}

Every Adinkra has an underlying ranked chromotopology (just forget about the dashings). However, not all ranked chromotopologies can be made into an Adinkra. A chromotopology is called \emph{adinkraizable} if it is the underlying chromotopology of some Adinkra. We define an \emph{ARC} to be an adinkraizable ranked chromotopology, which is exactly the data of an Adinkra minus the dashings. We say that an ARC (or Adinkra) is \emph{color-unordered} if it is \textbf{not} equipped with an ordering of the $n$ colors; one may think of a color-unordered ARC as an equivalence class of ARC's where two ARC's are considered equivalent if we can permute colors to get from one to the other. An alternative is to think of a color-unordered ARC as an ARC that only knows about the \textbf{partition} of its edges into the $n$ colors. We can tweak Lemma~\ref{lem:main equivalence} to suit alternate constraints:
\begin{itemize}
\item Forgetting about signs of the matrices in the construction of Lemma~\ref{lem:main equivalence} is the same as forgetting about the dashings of the Adinkras, which gives ARC's instead of Adinkras..
\item Counting equivalence classes of lists of generators by allowing the $S_d \times S_d$ action of separately permuting the rows and columns of the matrices is the same as forgetting about the orderings of the vertices in each row of the row-ordered valise Adinkra, which gives (non-row-ordered!) valise Adinkras.
\item Counting sets of generators instead of lists of generators is the same as forgetting about the order of the generators, which in turn is the same as forgetting about the ordering of the colors of an Adinkra/ARC (in other words, counting color-unordered Adinkras/ARC's).
\end{itemize}

Recall from Section~\ref{sec:introduction} that the authors of \cite{chappell2014adinkra}, \cite{gates-counting}, and \cite{randall} are interested in signed (resp. unsigned) sets of generators of $GR(d,N)$. However, they also do not care about row and column permutations of the matrices, as those do not really change the underlying physics. By the discussion above, this amounts to counting $N$-dimensional color-unordered valise Adinkras (resp. ARC's) with $2d$ vertices; that will be our goal in Section~\ref{sec:proof}.

We end with few important ideas from \cite{zhang:adinkras}. The intuition behind these ideas were already implicit from early works in the field.
\begin{itemize}
\item Adinkraizable chromotopologies are in bijection with \emph{$(N,k)$- doubly-even codes}; that is, for every $N$-dimensional chromotopology $G$, there is exactly one $k$-dimensional $\ZZ_2$-subspace of $\ZZ_2^N$ such that every vector has the number of $1$'s divisible by $4$. We call $k$ the \emph{code-dimension} of the corresponding $GR(d,N)$ algebra / chromotopology. 
\item If $C$ is the $(N,k)$ doubly-even code associated with an Adinkra, the vertices of the chromotopology are in bijection with cosets of $\ZZ_2^N$ of form $c+C$, where $2$ cosets have an edge connecting them if and only if there exist coset representatives $c_1$ and $c_2$ respectively (i.e. the cosets can be rewritten $c_1 +C$ and $c_2+C$) with $c_1$ and $c_2$ having Hamming distance $1$. When $C$ is trivial, this definition simply recovers the Haming cube. Note that this bijection implies that the chromotopology must have $2d = 2^{N-k}$ vertices, a power of $2$; this is not an obvious fact from the definitions of chromotopologies.
\end{itemize}

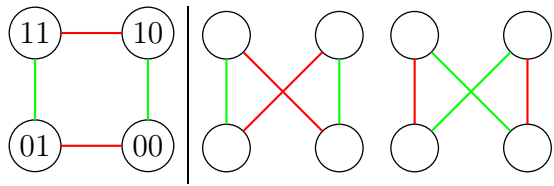
\begin{figure}[htb]
\begin{center}

\begin{tabular}{c|cc}
\begin{tikzpicture}[scale=0.15]
\SetUpEdge[labelstyle={draw}]
\Vertex[x=0,y=0]{01}
\Vertex[x=0,y=10]{11}
\Vertex[x=10,y=0]{00}
\Vertex[x=10,y=10]{10}
\Edge[color=red](10)(11)
\Edge[color=red](00)(01)
\Edge[color=green](00)(10)
\Edge[color=green](01)(11)
\end{tikzpicture} 
&
\begin{tikzpicture}[scale=0.15]
\SetVertexNoLabel
\SetUpEdge[labelstyle={draw}]
\Vertex[x=0,y=0]{01}
\Vertex[x=0,y=10]{11}
\Vertex[x=10,y=10]{00}
\Vertex[x=10,y=0]{10}
\Edge[color=red](10)(11)
\Edge[color=red](00)(01)
\Edge[color=green](00)(10)
\Edge[color=green](01)(11)
\end{tikzpicture} 
&
\begin{tikzpicture}[scale=0.15]
\SetVertexNoLabel
\SetUpEdge[labelstyle={draw}]
\Vertex[x=0,y=0]{01}
\Vertex[x=10,y=10]{11}
\Vertex[x=0,y=10]{00}
\Vertex[x=10,y=0]{10}
\Edge[color=red](10)(11)
\Edge[color=red](00)(01)
\Edge[color=green](00)(10)
\Edge[color=green](01)(11)
\end{tikzpicture} 
\end{tabular}
\caption{Left: the unique chromotopology (with Hamming-cube induced labels to aid the reader; we stress that the labels are \textbf{not} part of the data considered in this paper) for $N=2$ and $d=2$. Right: the $2$ different row-ordered valise ARC's we can obtain with this chromotopology. They belong to the same equivalence class under color permutation, so there is only one color-unordered row-ordered valise ARC. \label{fig:N=2 d=2}}
\end{center}
\end{figure}

\begin{ex}
\label{ex:N=2}
Consider $N = 2$ and $d = 2$, which corresponds to $k = 0$. Here, the code is trivial as it is $0$-dimensional, so the underlying chromotopology is the Hamming $2$-cube. We observe that there are $2$ different row-ordered valise ARC's, as seen in Figure~\ref{fig:N=2 d=2}. After picking an order of colors (say, $L_1$ corresponding to green and $L_2$ corresponding to red), these two row-ordered ARC's correspond to the two different possible lists of generators $(|L_1|, |L_2|)$, which are $(\begin{bmatrix}1 & 0 \\ 0 & 1\end{bmatrix}, \begin{bmatrix}0 & 1 \\ 1 & 0\end{bmatrix})$ and $(\begin{bmatrix}0 & 1 \\ 1 & 0\end{bmatrix}, \begin{bmatrix}1 & 0 \\ 0 & 1\end{bmatrix})$. As an example, the left row-ordered valise ARC has color $1$ (green) matching the first (resp. second) vertex of the top row with the first (resp. second) vertex of the bottom row, so $|L_1|$ is the identity matrix for that chromotopology.  For each row-ordered valise ARC, there are $8$ possible dashings to make them into row-ordered valise Adinkras, corresponding to the fact that there $8$ ways to obtain a list of generators $(L_1, L_2)$ from an unsigned list of generators $(|L_1|, |L_2|)$. See Figure~\ref{fig:N=2 dashings}. Note that the two ARC's are equivalent under exchanging the two colors. This corresponds to the fact that the two corresponding unsigned lists of generators $(|L_1|, |L_2|)$ are permutations of each other, so there is exactly $1$ unsigned set of generators $\{|L_1|, |L_2|\}$ and $8$ (signed) sets of generators $\{L_1, L_2\}$.
\end{ex}

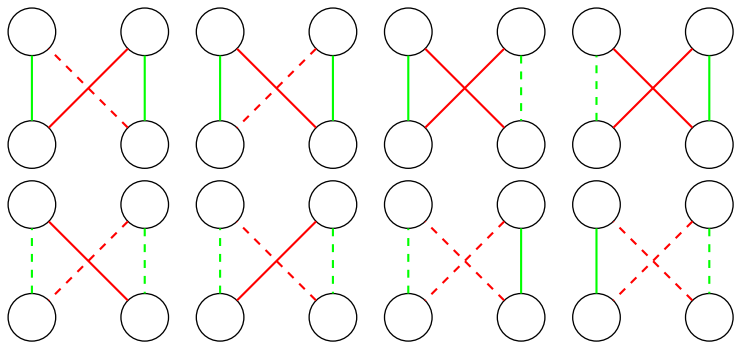
\begin{figure}[htb]
\begin{center}

\begin{tabular}{cccc}
\begin{tikzpicture}[scale=0.15]
\SetVertexNoLabel
\SetUpEdge[labelstyle={draw}]
\Vertex[x=0,y=0]{01}
\Vertex[x=0,y=10]{11}
\Vertex[x=10,y=10]{00}
\Vertex[x=10,y=0]{10}
\Edge[color=red, style=dashed](10)(11)
\Edge[color=red](00)(01)
\Edge[color=green](00)(10)
\Edge[color=green](01)(11)
\end{tikzpicture} 
&
\begin{tikzpicture}[scale=0.15]
\SetVertexNoLabel
\SetUpEdge[labelstyle={draw}]
\Vertex[x=0,y=0]{01}
\Vertex[x=0,y=10]{11}
\Vertex[x=10,y=10]{00}
\Vertex[x=10,y=0]{10}
\Edge[color=red](10)(11)
\Edge[color=red, style=dashed](00)(01)
\Edge[color=green](00)(10)
\Edge[color=green](01)(11)
\end{tikzpicture} 
& 
\begin{tikzpicture}[scale=0.15]
\SetVertexNoLabel
\SetUpEdge[labelstyle={draw}]
\Vertex[x=0,y=0]{01}
\Vertex[x=0,y=10]{11}
\Vertex[x=10,y=10]{00}
\Vertex[x=10,y=0]{10}
\Edge[color=red](10)(11)
\Edge[color=red](00)(01)
\Edge[color=green, style=dashed](00)(10)
\Edge[color=green](01)(11)
\end{tikzpicture} 
& 
\begin{tikzpicture}[scale=0.15]
\SetVertexNoLabel
\SetUpEdge[labelstyle={draw}]
\Vertex[x=0,y=0]{01}
\Vertex[x=0,y=10]{11}
\Vertex[x=10,y=10]{00}
\Vertex[x=10,y=0]{10}
\Edge[color=red](10)(11)
\Edge[color=red](00)(01)
\Edge[color=green](00)(10)
\Edge[color=green, style=dashed](01)(11)
\end{tikzpicture} 
\\
\begin{tikzpicture}[scale=0.15]
\SetVertexNoLabel
\SetUpEdge[labelstyle={draw}]
\Vertex[x=0,y=0]{01}
\Vertex[x=0,y=10]{11}
\Vertex[x=10,y=10]{00}
\Vertex[x=10,y=0]{10}
\Edge[color=red](10)(11)
\Edge[color=red, style=dashed](00)(01)
\Edge[color=green,style=dashed](00)(10)
\Edge[color=green,style=dashed](01)(11)
\end{tikzpicture} 
& 
\begin{tikzpicture}[scale=0.15]
\SetVertexNoLabel
\SetUpEdge[labelstyle={draw}]
\Vertex[x=0,y=0]{01}
\Vertex[x=0,y=10]{11}
\Vertex[x=10,y=10]{00}
\Vertex[x=10,y=0]{10}
\Edge[color=red, style=dashed](10)(11)
\Edge[color=red](00)(01)
\Edge[color=green,style=dashed](00)(10)
\Edge[color=green,style=dashed](01)(11)
\end{tikzpicture} 
& 
\begin{tikzpicture}[scale=0.15]
\SetVertexNoLabel
\SetUpEdge[labelstyle={draw}]
\Vertex[x=0,y=0]{01}
\Vertex[x=0,y=10]{11}
\Vertex[x=10,y=10]{00}
\Vertex[x=10,y=0]{10}
\Edge[color=red, style=dashed](10)(11)
\Edge[color=red, style=dashed](00)(01)
\Edge[color=green](00)(10)
\Edge[color=green,style=dashed](01)(11)
\end{tikzpicture} 
& 
\begin{tikzpicture}[scale=0.15]
\SetVertexNoLabel
\SetUpEdge[labelstyle={draw}]
\Vertex[x=0,y=0]{01}
\Vertex[x=0,y=10]{11}
\Vertex[x=10,y=10]{00}
\Vertex[x=10,y=0]{10}
\Edge[color=red, style=dashed](10)(11)
\Edge[color=red, style=dashed](00)(01)
\Edge[color=green,style=dashed](00)(10)
\Edge[color=green](01)(11)
\end{tikzpicture} 
\end{tabular}

\caption{The $8$ row-ordered valise Adinkras for one of the row-ordered valise ARC's for $N=d=2$, obtained by adding a dashing. \label{fig:N=2 dashings}}
\end{center}
\end{figure}

\begin{ex}
Consider $N = 4$ and $d = 4$, which corresponds to $k = 1$. See Figure~\ref{fig:3cube valise} for one such valise Adinkra, with the $1$-dimensional code $C= \{0000, 1111\}$.  One possible choice of coset representatives under this code would be the $16/2 = 8$ bitstrings where the first coordinate is $0$, as reflected in the Figure. As every vertex has $4$ possible bits to change, each vertex has degree $4$, corresponding to one of $4$ colors. If we let the black, red, green, and blue colors correspond to $L_1, L_2, L_3, L_4$ respectively, the corresponding list of generators of $GR(4,4)$ is precisely what we encountered in Example~\ref{ex:matrices}. As an example, if we follow the black edge corresponding to $L_1$ from the fourth vertex on top labeled $0000$, we obtain $1000$. This is not one of our coset representatives, but because $0111 = 1000+1111$, it is in the coset $0111+C$, represented by our second vertex on the bottom. Thus, there is a $1$ in the $(4,2)$ spot of the matrix $L_1$.
\end{ex}
\begin{figure}[htb]
\begin{center}
\begin{tikzpicture}[scale=0.25]
\SetUpEdge[labelstyle={draw}]
\Vertex[x=0,y=0]{0111}
\Vertex[x=0,y=10]{0101}
\Vertex[x=20,y=0]{0010}
\Vertex[x=20,y=10]{0000}
\Vertex[x=-10,y=10]{0110}
\Vertex[x=10,y=10]{0011}
\Vertex[x=-10,y=0]{0100}
\Vertex[x=10,y=0]{0001}
\Edge[color=red,style=dashed](0100)(0101)
\Edge[color=red,style=dashed](0000)(0001)
\Edge[color=red](0010)(0011)
\Edge[color=red](0110)(0111)
\Edge[color=green](0000)(0100)
\Edge[color=green, style=dashed](0001)(0101)
\Edge[color=green](0010)(0110)
\Edge[color=green, style=dashed](0011)(0111)
\Edge[color=blue, style=dashed](0000)(0010)
\Edge[color=blue, style=dashed](0001)(0011)
\Edge[color=blue](0100)(0110)
\Edge[color=blue](0101)(0111)
\Edge[color=black](0100)(0011)
\Edge[color=black](0000)(0111)
\Edge[color=black](0010)(0101)
\Edge[color=black](0110)(0001)
\end{tikzpicture}
\caption{A row-ordered valise Adinkra corresponding to $N=d=4$, with vertex labels to aid the reader. The corresponding list of generators appears in Example~\ref{ex:matrices}. \label{fig:3cube valise}}
\end{center}
\end{figure}
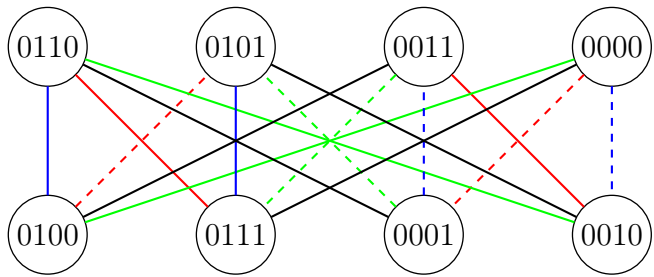

\begin{ex}
Consider $N = 8$ and $d = 8$, which corresponds to $k = 4$. One such code is the extended Hamming code $E_8$, which is the set of $16$ vectors generated by $\{11110000,00111100,00001111,01010101\}$. Then the $2d = 16$ vertices of the chromotopology can be indexed by $\{v+E_8\}$. As there are $2^8$ bitstrings in $\ZZ_2^8$ and $2^4$ bitstrings in $E_8$, we can select $2^{8-4} = 16$ coset representatives. Then we have edges between e.g. $00000000+E_8$ and $11110001+E_8$ because $00000000$ is a representative of the first coset and $11110001 + 11110000 = 00000001$ is a representative of the second coset, and those have Hamming distance $1$ between them. If we were to draw such an adinkra, each vertex would now have $8$ edges incident to it of different colors, as $N=8$.
\end{ex}

\section{Main Theorem and Proof}
\label{sec:proof}

We now present our main results, where we count equivalence classes (under row and column permutation) of sets (resp. unsigned sets) of generators of $GR(d,N)$, which is equivalent to counting $N$-dimensional color-unordered valise Adinkras (resp. ARC's) with $2d$ vertices. We use the ideas from \cite{zhang:adinkras} above along with some elementary combinatorics. Our goal is to reproduce the efforts of \cite{gates-counting} and \cite{randall} using the mathematical foundations we built from \cite{zhang:adinkras} to create a streamlined general approach that avoids case-analysis.

First, define $C(N,k)$ to be the number of doubly-even $(N,k)$ codes. We stress that the enumeration of $C(N,k)$ is known. Thus, one merit of our general approach is encapsulating a difficult, but solved, part of the problem so we do not have to end up repeating some of the work for specific cases. One can see \cite{gaborit1996mass}, which gives different formulae for different cases (we omit all but two cases for relevance to our goals of confirming earlier computations):

\begin{thm}[Theorem 7, \cite{gaborit1996mass}]
\label{thm:gaborit}
For all $N$, $C(N, 0) = 1$. For $1 \leq k \leq n/2$ and $N \geq 4$, we have:
\begin{itemize}
\item If $N = 4 \pmod{8}$, 
\[
C(N,k) = [\prod_{i=0}^{k-2} \frac{2^{N - 2i - 2} - 2^{N/2 - i - 1} - 2}{2^{i+1} - 1}] \cdot [\frac{1}{2^{k-1}} + \frac{2^{N-2k} - 2^{N/2 - k} - 2}{2^k - 1}].
\]
\item If $N = 0 \pmod{8}$, 
\[
C(N,k) = [\prod_{i=0}^{k-2} \frac{2^{N - 2i - 2} + 2^{N/2 - i - 1} - 2}{2^{i+1} - 1}] \cdot [\frac{1}{2^{k-1}} + \frac{2^{N-2k} + 2^{N/2 - k} - 2}{2^k - 1}].
\]
\item If $N = 1, 7 \pmod{8}$, etc.
\end{itemize}
\end{thm}

\begin{thm}
\label{thm:main}
For any $d$ and $N$, there are 
\[
\frac{(d)!(d - 1)! C(N, N-1-\log_2{d})}{N!}
\]
equivalence classes of sets $\{|L_1|, \ldots, |L_N|\}$ under row and column permutation.
\end{thm}

\begin{proof}
By our work in Section~\ref{sec:prelim} stemming from Lemma~\ref{lem:main equivalence}, this is equivalent to counting color-unordered valise ARC's with $N$ colors and $2d$ vertices. This means the ARC is $(N,k)$ with $2d = 2^{N-k}$, so $k = N-1-\log_2(d)$ is the dimension of the code. There are $C(N,k)$ doubly-even codes; fix such a double-even code $C$, which gives a chromotopology. Our strategy is to first count row-ordered valise ARC's with this chromotopology under a particular labeling scheme, then divide out by the symmetries.

Now, we can pick a single one of the $2d$ vertices to be labeled $00\cdots 00 + C$. Each permutation of the $N$ colors corresponds to an assignment of one of the $N$ indices to each color. This uniquely determines the label of all the other vertices, defining e.g. $10\cdots 00 + C$ to be the vertex connected to $00\cdots 00 + C$ via the first color, $01\cdots 00 + C$ to be the vertex connected via the second, and so forth.

We now have to pick the vertices to put on the top and the bottom of the valise. There are $2$ choices of whether $00\cdots 00 + C$ would be on the top or the bottom. Afterwards, there are $d!$ ways of arranging the vertices on the top and $d!$ ways of arranging the vertices on the bottom, which we must account for as we are counting row-ordered ARC's.

However, the final ARC does not know about these labels. As there are $2d$ labels, there's a $2d$-fold symmetry on which vertex we would have labeled $00\cdots 00 + C$, so we have to divide the answer by $2d$. As there are $N!$ ways to permute the colors, which we do not care about, we should also divide by $N!$. This gives a total of

\[
\frac{ C(N,k) (2) (d!) (d!)}{2d N!} =  \frac{(d)!(d - 1)! C(N, k)}{N!}
\]
configurations, as desired.
\end{proof}

We now wish to count sets of actual generators $\{L_1, \ldots, L_N\}$, not just unsigned sets of generators. This corresponds to allowing negative signs for the matrices and allowing dashings for the ARC's (i.e. turning them into Adinkras). Luckily, this is easy thanks to our previous result from another paper, which we proved in \cite{zhang:adinkras} with some elementary algebraic topology:
\begin{thm}[Theorem 5.8, \cite{zhang:adinkras}]
\label{thm:general dashings}
The number of dashings of an $(N,k)$-chromotopology $A$ is $2^{2^{N-k}+k-1}$.
\end{thm}
The main insight this result gives is that after a $k$ is selected, the number of dashings for all chromotopologies with the same $k$ are the same. Thus, we obtain: 
\begin{cor}
\label{cor:main}
For any $d$ and $N$, there are 
\[
\frac{(d)!(d - 1)! C(N, N-1-\log_2{d})2^{2d + N-2-\log_2(d)}}{N!}
\]
equivalence classes of sets $\{L_1, \ldots, L_N\}$ under row and column permutation. 
\end{cor}
\begin{proof}
This time, we are interested in $N$-dimensional color-unordered valise Adinkras (instead of ARC's) with $2d$ vertices. As Adinkras are just ARC's with a dashing, we simply multiply with the results of Theorems~\ref{thm:general dashings} and \ref{thm:main}.
\end{proof}

We now check that we indeed recover the counts of our examples and of previous research. 

\begin{ex}
For $N = d = 2$, we have $k = 0$. Theorem~\ref{thm:gaborit} (or direct observation) produces $C(2,0) = 1$. Thus, we get $(2!)(1!)/(2!) = 1$ set of $\{|L_1|, |L_2|\}$, as we saw previously in Example~\ref{ex:N=2}. Multiplying by $2^{4+2-2-1} = 8$, we get $8$ sets of $\{L_1, L_2\}$ of dashings via Corollary~\ref{cor:main} (or directly, via Theorem~\ref{thm:general dashings}), which matches Figure~\ref{fig:N=2 dashings}.  
\end{ex}

\begin{ex}
\label{ex:N=4 final}
For $N = d = 4$, we have $k = 1$. Theorem~\ref{thm:gaborit} produces $C(4,1) = (\frac{1}{1} + \frac{2^2 - 2^1 - 2}{1}) = 1$ doubly-even code. Thus, we get $(4!) (3!)/4! = 6$ equivalence classes of  sets of $\{|L_1|, \ldots, |L_4|\}$. Multiplying by $2^{8+4-2-2} = 2^{8}$, we obtain $1536$ equivalence classes of sets of $\{L_1, \ldots, L_4\}$ via Corollary~\ref{cor:main}.
\end{ex}

\begin{ex}
\label{ex:N=8 final}
For $N = d = 8$, we have $k = 4$. Theorem~\ref{thm:gaborit} produces 
\[
C(8,4) = (\frac{2^{8-2} + 2^{4-1} - 2}{2-1}) (\frac{2^{8-4} + 2^{4-2} - 2}{2^2-1})(\frac{2^{8-6} + 2^{4-3} - 2}{2^3-1})(\frac{1}{8} + \frac{1 + 1 - 2}{15}) = 30.
\] 
Thus, we get $(8!)(7!)(30)/8! = 151200$ equivalence classes of sets of $\{|L_1|, \ldots, |L_8|\}$. Multiplying by $2^{16+8-2-3} = 2^{19}$, we obtain $79272345600$ equivalence classes of sets of $\{L_1, \ldots, L_8\}$.
\end{ex}

The numbers in Examples~\ref{ex:N=4 final} and \ref{ex:N=8 final} indeed match previous work in \cite{gates-counting}, as seen in Figure~\ref{fig:previous}.

\section*{Acknowledgments}
We wish to thank the authors of \cite{gates-counting}, namely Sylvester Gates, Tristan H\"ubsch, Kevin Iga, and Steven Mendez-Diez, for introducing the problem to us and providing helpful discussion in the spirit of collaborative mathematical discovery, despite using ``competing'' methods! We especially thank Kevin Iga for pointing us to Gaborit's work \cite{gaborit1996mass}.

\bibliographystyle{abbrv}
\bibliography{adinkras}

\end{document}